\documentclass[11pt]{amsart}
\usepackage{amsmath}
\usepackage{amssymb}
\usepackage{amsthm}
\usepackage{amscd}
\usepackage{xypic}
\usepackage{delarray}

\def\G{{\mathbb G}}

\def\P{{\mathbb P}}

\newtheorem{theorem}{Theorem}[section]
\newtheorem{lemma}[theorem]{Lemma}

\newtheorem{corollary}[theorem]{Corollary}

\theoremstyle{definition}

\newtheorem{remark}[theorem]{Remark}

\newtheorem{convention and reminder}[theorem]{Convention and Reminder}
\newtheorem{convention and remark}[theorem]{Convention and Remark}
\newtheorem{definition and remark}[theorem]{Definition and Remark}

\newtheorem{reminders and definition}[theorem]{Reminders and Definition}

\newtheorem{notation and remarks}[theorem]{Notation and Remarks}
\newtheorem{notation and remark}[theorem]{Notation and Remark}

\newcommand\Ker{\operatorname{\Ker}}

\def\QR{\mathsf{QR}}

\title[On the rank of quadratic equations for curves of high degree]{On the rank of quadratic equations for curves of high degree}

\author[E. Park]{Euisung Park}
\address{Euisung Park, Department of Mathematics, Korea University, Seoul 136-701, Republic of Korea}
\email{euisungpark@korea.ac.kr}

\begin{document}


\begin{abstract}
Let $\mathcal{C} \subset \P^r$ be a linearly normal curve of
arithmetic genus $g$ and degree $d$. In \cite{SD}, B. Saint-Donat proved that the homogeneous ideal
$I(\mathcal{C})$ of $\mathcal{C}$ is generated by quadratic equations of
rank at most $4$ whenever $d \geq 2g+2$. Also, in \cite{EKS} Eisenbud, Koh and Stillman
proved that $I(\mathcal{C})$ admits a determinantal presentation if $d \geq 4g+2$. In this paper, we will show that $I(\mathcal{C})$ can be generated by quadratic equations of rank $3$ if either
$g=0,1$ and $d \geq 2g+2$ or else $g \geq 2$ and $d \geq 4g+4$.
\end{abstract}

\maketitle \setcounter{page}{1}

\section{Introduction}

\noindent Throughout this paper, we work over an algebraically
closed field $\Bbbk$ of arbitrary characteristic. We denote by
$\mathbb P^r$ the projective $r$-space over $\Bbbk$.

Let $\mathcal{C}$ be a projective integral curve of arithmetic genus
$g$ and $\mathcal{L}$ a base point free line bundle on
$\mathcal{C}$ of degree $d$, defining a map
\begin{equation*}
\phi_{|\mathcal{L}|} : \mathcal{C} \rightarrow \P^r
\end{equation*}
where $r = h^0 (\mathcal{C},\mathcal{L})-1$. We denote by
$I(\mathcal{C})$ the homogeneous ideal of $\phi_{|\mathcal{L}|} (
\mathcal{C})$ in $\P^r$. The generating structure of $I(\mathcal{C})$ is relatively well understood for the cases where either
\begin{enumerate}
\item[$(i)$] $\mathcal{C}$ is smooth and $\mathcal{L}=\omega_{\mathcal{C}}$ is the canonical
line bundle
\end{enumerate}
or else
\begin{enumerate}
\item[$(ii)$] $d \geq 2g+1$.
\end{enumerate}
In case $(i)$, Max Noether-Enriques-Petri Theorem in \cite{SD2}
says that if $\mathcal{C}$ is not hyperelliptic, then the above map
is a projectively normal embedding, and $I(\mathcal{C})$ is
generated by quadrics except $\mathcal{C}$ is a plane quintic
($g=6$) or trigonal (cf. \cite{En, Pe}). In \cite{Gre}, M. Green
reproved the classical Torelli's Theorem by showing that
$I(\mathcal{C})$ is generated by quadrics of rank $3$ and $4$. See
also \cite{ArH}. In case $(ii)$, the above map is always a
projectively normal embedding (cf. \cite{C, M, F}). In \cite{SD}, B.
Saint-Donat proved that if $d \geq 2g+2$ then $I(\mathcal{C})$ is
generated by quadratic equations of rank at most $4$. Also Eisenbud,
Koh and Stillman in \cite{EKS} proved that if $d \geq 4g+2$ then
$(\mathcal{C},\mathcal{L})$ is \textit{determinantally presented} in
the sense that $I(\mathcal{C})$ is generated by $2$-minors of a
$1$-generic matrix of linear forms on $\P^r$. Concerned with these works, we will say that $(\mathcal{C},\mathcal{L})$ \textit{satisfies property $\QR (k)$} for a positive integer $k$ if $\mathcal{L}$ is very ample and $I(\mathcal{C})$ can be generated by quadrics of rank at most $k$. For example,
$(\mathcal{C},\mathcal{L})$ satisfies property $\QR (4)$ if either
$d \geq 2g+2$ (cf. \cite{SD}) or else $\mathcal{L}=K_{\mathcal{C}}$
and $I(\mathcal{C})$ is generated by quadrics (cf. \cite{Gre}).

In this paper, we obtain the following new result about
the quadratic generators of the homogeneous ideal $I(\mathcal{C})$.

\begin{theorem}\label{thm:main 1}
Let $\mathcal{C}$ be a projective integral curve of arithmetic genus
$g$ and $\mathcal{L}$ a line bundle on $\mathcal{C}$ of degree $d \geq 2g+2$.
Then $(\mathcal{C},\mathcal{L})$ satisfies property $\QR (3)$ whenever either
\begin{enumerate}
\item[$(i)$] ${\rm char}~\Bbbk \neq 2$, $g = 0$ and $d \geq 2$, or else
\item[$(ii)$] ${\rm char}~\Bbbk \neq 2$, $g = 1$ and $d \geq 4$, or else
\item[$(iii)$] ${\rm char}~\Bbbk \neq 2,3$, $g \geq 2$ and $d \geq 4g+4$, or else
\item[$(iv)$] ${\rm char}~\Bbbk = 0$, $g \geq 7$, $\mathcal{C}$
is smooth and general in the moduli $\mathcal{M}_g$ and $d \geq 4g+2$.
\end{enumerate}
\end{theorem}

For the proofs of these statements, see Corollary \ref{cor:elliptic nor curve}, Corollary \ref{cor:rat nor curve}, Theorem \ref{thm:main 2} and Theorem \ref{thm:main 3}.

Roughly speaking, this result says that all sufficiently high degree
embedding of projective integral curves share the property $\QR
(3)$. Also note that $I(\mathcal{C})$ contains no quadrics of rank $<3$ since $\phi_{|\mathcal{L}|} (
\mathcal{C})$ is nondegenerate in $\P^r$.

Theorem \ref{thm:main 1} is a consequence of the following interesting result on the property $\QR
(3)$ of projective integral curves.

\begin{theorem}\label{thm:curve structure of rank}
Let $\mathcal{C}$ be a projective integral curve of arithmetic genus
$g$ over an algebraically closed field $\Bbbk$ with ${\rm
char}~\Bbbk \neq 2$. Suppose that there exists an integer $\tau \geq
2g+2$ satisfying the condition that
\begin{enumerate}
\item[$(\dagger)$] every line bundle $\mathcal{L}$ on $\mathcal{C}$ of degree $\tau$ satisfies property $\QR(3)$.
\end{enumerate}
Then every line bundle on $\mathcal{C}$ of degree $\geq \tau$
satisfies property $\QR(3)$.
\end{theorem}

For example, it is obviously true that $(\P^1 , \mathcal{O}_{\P^1} (2))$ satisfies property $\QR
(3)$. Then we can conclude by Theorem \ref{thm:curve structure of rank} that $(\P^1 , \mathcal{O}_{\P^1} (\ell))$ satisfies property $\QR (3)$ for all $\ell \geq 2$. In this way, we were able to prove the statements of Theorem \ref{thm:main 1} one by one using Theorem \ref{thm:curve structure of rank}.

\begin{remark}
$(1)$ Theorem \ref{thm:curve structure of rank} leads us to define a new invariant of projective integral curves. To be precise, let $\tau (\mathcal{C})$ be the
smallest integer $\tau \geq 2g+2$ such that every line bundle
$\mathcal{L}$ on $\mathcal{C}$ of degree $\tau$ satisfies property
$\QR(3)$. By Theorem \ref{thm:main 1} and Theorem \ref{thm:curve structure of rank}, we know the followings:\\
\begin{enumerate}
\item[$(i)$] When $g=0,1$, it holds that $\tau (\mathcal{C}) = 2g+2$.
\item[$(ii)$] When $g \geq 2$, it holds that $2g+2 \leq \tau (\mathcal{C}) \leq 4g+4$. In particular, $\tau
(\mathcal{C})$ does exist.
\item[$(iii)$] Every line bundle on $\mathcal{C}$ of degree $\geq \tau (\mathcal{C})$
satisfies property $\QR(3)$.\\
\end{enumerate}

\noindent It will be a natural problem to find an upper bound of $\tau (\mathcal{C})$ which is better than the number $4g+4$ obtained in Theorem \ref{thm:main 1}.

\noindent $(2)$ Let $\mathcal{C}$ be a projective complex smooth curve of genus $g \geq 3$ which is not  a hyperelliptic, a plane quintic or a trigonal curve. Then its canonical embedding $\mathcal{C} \subset \P^{g-1}$ satisfies property $\QR(4)$ (cf. \cite{Gre}). By \cite[Example 5.4 and 5.8]{HLMP}, infinitely many canonical curves satisfy property $\QR(3)$. On the other hand, there is a canonical curve $\mathcal{C} \subset \P^5$ of genus $6$ which fails to satisfy property $\QR(3)$. For details, see \cite[Example 6.2]{HLMP}.
\end{remark}

\section{Proof of Theorem \ref{thm:curve structure of rank}}

\noindent This section is devoted to giving a proof of Theorem \ref{thm:curve structure of
rank}. We begin with a few elementary facts.

\begin{notation and remarks}\label{nar:curve inner projection}
Let $\mathcal{C}$ be a projective integral curve of arithmetic genus
$g$ and let $\mathcal{L}$ be a line bundle on $\mathcal{C}$ of degree $d
\geq 2g+3$, defining an embedding
\begin{equation*}
\mathcal{C}_0 := \phi_{|\mathcal{L}|} (\mathcal{C}) \rightarrow \P H^0 (\mathcal{C},\mathcal{L}) = \P^{d-g} .
\end{equation*}
Also let $p_1$ and $p_2$ be two distinct smooth points of $\mathcal{C}$. We denote by $\mathbb{L}$ the line through $p_1$ and $p_2$ in $\P^{d-g}$.

\noindent $(1)$ Note that the line bundles $\mathcal{L}$, $\mathcal{L}(-p_1)$, $\mathcal{L}(-p_2)$ and
$\mathcal{L}(-p_1 -p_2)$ are very ample and define projectively
normal embedding (cf. \cite{C, M, F}). Furthermore, $\mathcal{L}$,
$\mathcal{L} (-p_1)$ and $\mathcal{L}(-p_2 )$ satisfy property $\QR(4)$ (cf. \cite{SD}).
\smallskip

\noindent $(2)$ Consider the cones
\begin{equation*}
S_1 = \mbox{Join}(p_1 , \mathcal{C}_0 ),~S_2 = \mbox{Join}(p_2 ,
\mathcal{C}_0 ), ~T = \mbox{Join}(\mathbb{L} , \mathcal{C}_0 ) \subset
\P^{d-g}.
\end{equation*}
The bases of the cones $S_1$, $S_2$ and $T$ can be regarded respectively as the
linearly normal curves
\begin{equation*}
\mathcal{C}_1 := \phi_{|\mathcal{L}(-p_1 )|} (\mathcal{C}) \subset \P H^0 (\mathcal{L} (-p_1 )) = \P^{d-g-1} ,
\end{equation*}
\begin{equation*}
\mathcal{C}_2 := \phi_{|\mathcal{L}(-p_2 )|} (\mathcal{C}) \subset \P H^0 (\mathcal{L} (-p_2 )) = \P^{d-g-1} ,
\end{equation*}
and
\begin{equation*}
\mathcal{C}_{12}:= \phi_{|\mathcal{L}(-p_1 -p_2 )|} (\mathcal{C}) \subset \P H^0 (\mathcal{L} (-p_1 - p_2 ))= \P^{d-g-2}.
\end{equation*}
\smallskip

\noindent $(3)$ Let the homogeneous ideals of $\mathcal{C}_0$, $S_1$, $S_2$
and $T$ in $\P^{d-g}$ be respectively $I(\mathcal{C }_0 )$, $I(S_1 )$, $I(S_2 )$ and $I(T)$. There are the following inclusions:
\begin{equation*}
I(\mathcal{C }_0 )_2 ~\supseteq~  V := I( S_1  )_2 + I( S_2)_2
~\supseteq~  I( S_1  )_2 \cap I( S_2)_2 ~\supseteq~  I(T)_2
\end{equation*}
\end{notation and remarks}

\begin{lemma}\label{lem:dimension counting}
Keep the notations in Notation and Remarks \ref{nar:curve inner
projection}. Then
\begin{equation*}
I( S_1  )_2 \cap I( S_2)_2 =  I( T )_2 .
\end{equation*}
and
\begin{equation*}
{\rm dim}_{\Bbbk}~V = {\rm dim}_{\Bbbk}~ I(\mathcal{C}_0)_2 -1 .
\end{equation*}
\end{lemma}

\begin{proof}
To show the first equality, let us suppose that there is a quadric $Q \in I( S_1  )_2 \cap I( S_2)_2 - I( T )_2$. Then $T \cap Q$ contains $S_1$ and $S_2$ while
\begin{equation*}
2 \times \deg (T) = 2(d-2) < 2(d-1) = \deg (S_1 ) + deg (S_2 ).
\end{equation*}
This is a contradiction (cf. \cite[Theorem 7.7]{H}). In consequence, the desired equality $I( S_1  )_2 \cap I( S_2)_2 =  I( T )_2$ holds.

To show the second equality, consider the elementary formula
\begin{equation*}
{\rm dim}_{\Bbbk}~V = {\rm dim}_{\Bbbk}~I(S_1)_2 + {\rm
dim}_{\Bbbk}~I(S_2)_2 - {\rm dim}_{\Bbbk}~ I(S_1)_2 \cap I(S_2)_2 .
\end{equation*}
By Notation and Remarks \ref{nar:curve inner projection}.(2), we
have
\begin{equation*}
{\rm dim}_{\Bbbk}~I(S_i )_2 = {\rm dim}_{\Bbbk}~I(\mathcal{C}_i )_2
\end{equation*}
for each $i=1,2$ and
\begin{equation*}
{\rm dim}_{\Bbbk}~ I(S_1)_2 \cap I(S_2)_2 = {\rm dim}_{\Bbbk}~I(T
)_2 = {\rm dim}_{\Bbbk}~I(\mathcal{C}_{12} )_2 .
\end{equation*}
Also, $\mathcal{C}_0 \subset \P^{d-g}$, $\mathcal{C}_1 ,
\mathcal{C}_2 \subset \P^{d-g-1}$ and $\mathcal{C}_{12} \subset
\P^{d-g-2}$ are projectively normal curves and hence it holds that
\begin{equation*}
{\rm dim}_{\Bbbk}~I(\mathcal{C}_0 )_2 = {d-g \choose 2} -g ,
\end{equation*}
\begin{equation*}
{\rm dim}_{\Bbbk}~I(\mathcal{C}_1 )_2 = {\rm
dim}_{\Bbbk}~I(\mathcal{C}_2 )_2 = {d-g -1 \choose 2} -g
\end{equation*}
and
\begin{equation*}
{\rm dim}_{\Bbbk}~I(\mathcal{C}_{12} )_2 = {d-g-2 \choose 2} -g .
\end{equation*}
Therefore it follows that ${\rm dim}_{\Bbbk}~V = {d-g \choose 2}
-g-1 = {\rm dim}_{\Bbbk}~ I(\mathcal{C}_0) -1$.
\end{proof}

Now, we are ready to prove Theorem \ref{thm:curve structure of rank}.\\

\noindent {\bf Proof of Theorem \ref{thm:curve structure of rank}.}
Firstly, note that it suffices to prove the following statement $(\ast)$.\\

\begin{enumerate}
\item[$(\ast)$] $\mathcal{L}$ satisfies property $\QR(3)$ for every line bundle $\mathcal{L}$ of degree
$\tau +1$.\\
\end{enumerate}

\noindent Let $\mathcal{L}$ be a line bundle on $\mathcal{C}$ of degree $\tau +1$, defining a linearly normal embedding
\begin{equation*}
\mathcal{C}_0 \subset \P H^0 (\mathcal{L}) = \P^{\tau +1-g}.
\end{equation*}
Choose three distinct smooth points $p_1$, $p_2$ and $p_3$ of $\mathcal{C}$. For $i=1,2,3$, let
\begin{equation*}
S_i := \mbox{Join}(p_i , \mathcal{C}_0 )  \subset  \P^{\tau+1 -g}.
\end{equation*}
So, a general hyperplane section of $S_i$ is isomorphic to the
linearly normal curve
\begin{equation*}
\mathcal{C}_i \subset \P H^0 (\mathcal{L} (-p_i )) = \P^{\tau -g}
\end{equation*}
of degree $\tau$ for each $i=1,2,3$. By our assumption, $I(\mathcal{C}_i)$ can be
generated by quadrics of rank $3$. Therefore $I(S_i )$ can be also generated by quadrics of rank $3$ since $S_i$ is a cone having $\mathcal{C}_i$ as a base. Now, we
will show that
\begin{equation}\label{eq:claim on spaning}
I(\mathcal{C}_0 )_2 = \langle I(S_1)_2 , I(S_2)_2 , I(S_3)_2
\rangle.
\end{equation}
Obviously, this completes the proof of $(\ast)$. To prove
(\ref{eq:claim on spaning}), first recall that
\begin{equation*}
{\rm dim}_{\Bbbk} ~ \langle I(S_1)_2 , I(S_2)_2 \rangle = {\rm
dim}_{\Bbbk} ~  I(\mathcal{C}_0)_2 -1
\end{equation*}
by Lemma \ref{lem:dimension counting}. Thus it suffices to check
that $I(S_3)_2$ is not contained in $\langle I(S_1)_2 , I(S_2)_2
\rangle$. Indeed, if $I(S_3)_2$ is a subspace of $\langle I(S_1)_2 ,
I(S_2)_2 \rangle$ then it holds that
\begin{equation*}
S_1 \cap S_2 \subset S_3.
\end{equation*}
In particular, the line $\langle p_1 , p_2 \rangle$ is contained in
$S_1 \cap S_2$ and hence in $S_3$. Then $\langle p_1 , p_2
, p_3 \rangle$ is also contained in $S(p_3 )$. Note that $\langle
p_1 , p_2 , p_3 \rangle$ cannot be a line since $\mathcal{C}_0$ is
cut out by quadrics and so it does not admit a tri-secant line. Therefore the irreducible surface $S_3$
contains the plane $\langle p_1 , p_2 , p_3 \rangle$, which is a
contradiction. In consequence, $I(S_3)_2$ is not contained in
$\langle I(S_1)_2 , I(S_2)_2 \rangle$ and hence we have
\begin{equation*}
{\rm dim}_{\Bbbk} ~ \langle I(S_1)_2 , I(S_2)_2 , I(S_3)_2 \rangle
\geq {\rm dim}_{\Bbbk} ~ \langle I(S_1)_2 , I(S_2)_2 \rangle +1 = {\rm dim}_{\Bbbk} ~
\langle I(\mathcal{C}_0)_2 .
\end{equation*}
It follows that the desired equality in (\ref{eq:claim on spaning}) holds. This completes the proof of Theorem \ref{thm:curve structure of rank}.  \qed \\

We finish this section by applying Theorem \ref{thm:curve structure
of rank} to the rational normal curves and elliptic normal curves.

\begin{corollary}\label{cor:rat nor curve}
$(\P^1 , \mathcal{O}_{\P^1} (d))$ satisfies property $\QR(3)$ for
all $d \geq 2$.
\end{corollary}

\begin{proof}
When $g=0$, it is obviously true that $(\P^1 , \mathcal{O}_{\P^1}
(2))$ satisfies property $\QR(3)$. Then Theorem \ref{thm:curve
structure of rank} implies that $(\P^1 , \mathcal{O}_{\P^1} (d))$
satisfies property $\QR(3)$ for all $d \geq 2$.
\end{proof}

\begin{remark}
Corollary \ref{cor:rat nor curve} is firstly proved in \cite[Corollary
2.4]{HLMP} where a rank $3$ generators are explicitly given. Here we provide another proof, which shows that every rational normal curve satisfies property $\QR(3)$ since $(\P^1 , \mathcal{O}_{\P^1}
(2))$ corresponds to a plane conic curve defined by a quadric of rank $3$.
\end{remark}

\begin{corollary}\label{cor:elliptic nor curve}
Let $\mathcal{C}$ be a projective integral curve of arithmetic genus
$1$ and $\mathcal{L}$ a line bundle on $\mathcal{C}$. If $d \geq 4$, then $(\mathcal{C},\mathcal{L})$ satisfies property $\QR (3)$.
\end{corollary}

\begin{proof}
By Theorem \ref{thm:curve structure of rank}, it suffices to show that
$(\mathcal{C},\mathcal{L})$ satisfies property $\QR(3)$ for every
line bundle $\mathcal{L}$ on $\mathcal{C}$ of degree $4$. Now, let
$\mathcal{L}$ be a line bundle of degree $4$ on $\mathcal{C}$. Then
\begin{equation*}
\mathcal{C} \subset \P H^0 (\mathcal{L}) = \P^3
\end{equation*}
is a complete intersection of two quadrics, say $Q_1$ and $Q_2$. Let
$\P^5$ denote the projective space of quadrics in $\P^3$. Also let
$W \subset \P^5$ be the locus of quadrics of rank $\leq 3$. Then $W$
is a hypersurface of degree $4$.

When $\mathcal{C}$ is smooth, the line $\overline{Q_1 , Q_2}$ in $\P^5$ intersects with $W$
exactly at four points (cf. Proposition 22.34 in \cite{Harris}).
Those four intersection points correspond to quadratic equations of
rank $3$ in $I(\mathcal{C})$. This shows that
$(\mathcal{C},\mathcal{L})$ satisfies property $\QR(3)$.

Now, suppose that $\mathcal{C}$ is singular. By \cite[Theorem 1.1]{LPS},
we may assume that
\begin{equation*}
Q_1 = x_0 ^2 + x_1 x_2 \quad \mbox{and} \quad Q_2 = x_3 ^2 + x_0 x_2
\end{equation*}
if $\mathcal{C}$ has the cusp singularity, and
\begin{equation*}
Q_1 = x_0 ^2 + x_1 x_2 \quad \mbox{and} \quad Q_2 = x_3 ^2 + x_0 x_2
+ x_0 x_3
\end{equation*}
if $\mathcal{C}$ has the nodal singularity. This shows that
$(\mathcal{C},\mathcal{L})$ satisfies property $\QR(3)$ when
$\mathcal{C}$ is singular. As was mentioned at the beginning, this
completes the proof.
\end{proof}

\begin{remark}
In \cite[Theorem IV.1.3]{Hulek}, the author proved that every elliptic normal curve of degree $\geq 4$ is the scheme-theoretic intersection of the quadrics of rank $3$ which contains it. Corollary \ref{cor:elliptic nor curve} shows that an ideal-theoretic version of the statement of \cite[Theorem IV.1.3]{Hulek} holds also.
\end{remark}

\section{Proof of Theorem \ref{thm:main 1}}

\noindent Here we will complete the proof of Theorem \ref{thm:main 1} by verifying the following two statements.

\begin{theorem}\label{thm:main 2}
Suppose that ${\rm char}~\Bbbk \neq 2,3$. Let $\mathcal{C}$ be a projective integral curve of arithmetic genus
$g \geq 2$ and $\mathcal{L}$ a line bundle on $\mathcal{C}$ of degree $d$. If $d \geq 4g+4$, then $(\mathcal{C},\mathcal{L})$ satisfies property $\QR (3)$.
\end{theorem}

\begin{theorem}\label{thm:main 3}
Suppose that ${\rm char}~\Bbbk = 0$ and $\mathcal{C}$ is a smooth projective curve of genus $g \geq 7$ and general in the moduli $\mathcal{M}_g$. If $\mathcal{L}$ is a line bundle on $\mathcal{C}$ of degree $d \geq 4g+2$,
then $(\mathcal{C},\mathcal{L})$ satisfies property $\QR (3)$.
\end{theorem}

We begin with a few remarks.

\begin{remark}\label{rmk:curve important}
Let $\mathcal{C}$ be a projective integral curve of arithmetic genus $g \geq 1$.

\noindent $(1)$ The group $\mbox{Pic}^0 ( \mathcal{C})$ is known to be divisible. For details, we refer the
reader to see \cite[V.$\S~3$]{Se}. Indeed, let $\pi : \tilde{\mathcal{C}} \rightarrow \mathcal{C}$ be the normalization of $\mathcal{C}$ and let $\tilde{g}$ denote the genus of $\tilde{\mathcal{C}}$. Then the following sequence is exact:
\begin{equation*}
0 \rightarrow \pi_* \left( \mathcal{O}_{\tilde{\mathcal{C}}} ^* /
\mathcal{O}_{\mathcal{C}} ^* \right) \rightarrow \mbox{Pic}^0
(\mathcal{C}) \rightarrow \mbox{Pic}^0
(\tilde{\mathcal{C}})\rightarrow 0
\end{equation*}
The third term $\mbox{Pic}^0 (\tilde{\mathcal{C}})$ is an Abelian variety of dimension $\tilde{g}$. In particular, it is divisible. Also the first term can be written as
\begin{equation*}
\pi_* \left( \mathcal{O}_{\tilde{\mathcal{C}}} ^* /
\mathcal{O}_{\mathcal{C}} ^* \right) = \bigoplus_{P \in \mathcal{C}}
\tilde{\mathcal{O}}_P ^* /\mathcal{O}_P ^*
\end{equation*}
where $\tilde{\mathcal{O}}_P$ is the integral closure of
$\mathcal{O}_P$ in the function field of $\mathcal{C}$. It is known
that $\tilde{\mathcal{O}}_P ^* /\mathcal{O}_P ^*$ is always an
extension of several $\G_a$'s, $\G_m$'s and groups of Witt vectors.
In particular, it is a divisible group. In consequence, $\mbox{Pic}^0 (
\mathcal{C})$ is a divisible group.

\noindent $(2)$ Let $\mathcal{M}$ be a line bundle on $\mathcal{C}$ of degree $\geq 2g+2$. Then $(\mathcal{C},\mathcal{M})$ satisfies property $\QR (4)$ (cf. \cite{SD}). Now, by applying \cite[Corollary 5.2]{HLMP} to $\mathcal{M}$, it holds that if ${\rm char}~\Bbbk \neq 2,3$ then $(\mathcal{C},\mathcal{M}^2 )$ satisfies property $\QR (3)$.

\noindent $(3)$ Suppose that ${\rm char}~\Bbbk = 0$ and $\mathcal{C}$ is smooth. Let $\mathcal{M}$ be a line bundle on $\mathcal{C}$ of degree $2g+1$, which defines a linearly normal embedding
\begin{equation*}
\mathcal{C} \subset \P^{g+1} .
\end{equation*}
By \cite[Theorem 2]{GL}, the following two statements are equivalent:
\begin{enumerate}
\item[$(i)$] $I(\mathcal{C})$ fails to be generated by quadrics;
\item[$(ii)$] there is an effective divisor $D$ of degree $3$ such that $\mathcal{M} = \omega_\mathcal{C} \otimes \mathcal{O}_\mathcal{C} (D)$.
\end{enumerate}
\end{remark}

Now, we give the proofs of Theorem \ref{thm:main 2} and Theorem \ref{thm:main 3}.\\

\noindent {\bf Proof of Theorem \ref{thm:main 2}.} By Theorem \ref{thm:curve structure of rank}, we need to prove that $(\mathcal{C},\mathcal{L})$ satisfies property $\QR(3)$ when $\mathcal{L}$ is a line bundle of degree $4g+4$. To this aim, choose any line bundle $\mathcal{M}_0$ of degree $2g+2$. Then $\mathcal{L} \otimes \mathcal{M}_0 ^{-2}$ is a member of $\mbox{Pic}^0 ( \mathcal{C})$. In particular, there exists a line bundle $\mathcal{N} \in \mbox{Pic}^0 ( \mathcal{C})$ such that
\begin{equation*}
\mathcal{L} \otimes \mathcal{M}_0 ^{-2} =  \mathcal{N}^2
\end{equation*}
since $\mbox{Pic}^0 ( \mathcal{C})$ is a $2$-divisible group (cf. Remark \ref{rmk:curve important}.(1)). In consequence,
\begin{equation*}
\mathcal{L} = \mathcal{M}^2
\end{equation*}
where $\mathcal{M} := \mathcal{M}_0 \otimes \mathcal{N}$ is a line bundle of degree $2g+2$. Namely, $\mathcal{L}$ is the square of a line bundle of degree $2g+2$. Now, it follows by Remark \ref{rmk:curve important}.(2) that $(\mathcal{C},\mathcal{L})$ satisfies property $\QR (3)$. \qed \\

\noindent {\bf Proof of Theorem \ref{thm:main 3}.} We may assume that $\mathcal{C}$ is not hyperelliptic.
By Theorem \ref{thm:curve structure of rank}, we need to show that $(\mathcal{C},\mathcal{L})$
satisfies property $\QR(3)$ when $\mathcal{L}$ is of degree $4g+2$. For such a line bundle $\mathcal{L}$,
one can show that there exists a line bundle $\mathcal{M}$ of degree $2g+1$ such that
\begin{equation*}
\mathcal{L}   =  \mathcal{M}^2 .
\end{equation*}
For details, see the proof of Theorem \ref{thm:main 2}. Now, choose a non-trivial $2$-torsion line bundle $\mathcal{E}$ on $\mathcal{C}$. Let
\begin{equation*}
\mathcal{C}_1 \subset \P^{g+1} \quad \mbox{and} \quad \mathcal{C}_2 \subset \P^{g+1}
\end{equation*}
be respectively the linearly normal embedding of $\mathcal{C}$ by $\mathcal{M}$ and $\mathcal{M} \otimes \mathcal{E}$. We claim that at least one of $I(\mathcal{C}_1 )$ and $I(\mathcal{C}_2 )$ is generated by quadrics. Indeed, suppose not. Then, by Remark \ref{rmk:curve important}.(3), we can write
\begin{equation*}
\mathcal{M} = \omega_\mathcal{C} \otimes \mathcal{O}_\mathcal{C} (D)  \quad \mbox{and} \quad \mathcal{M} \otimes \mathcal{E} = \omega_\mathcal{C} \otimes
\mathcal{O}_\mathcal{C} (D')
\end{equation*}
for some effective divisors $D$ and $D'$ of degree $3$. Then, we get
\begin{equation*}
\mathcal{M} = \omega_\mathcal{C} \otimes \mathcal{O}_\mathcal{C} (D) = \omega_\mathcal{C} \otimes
\mathcal{O}_\mathcal{C} (D') \otimes \mathcal{E}^{-1}
\end{equation*}
and hence $\mathcal{O}_\mathcal{C} (D'-D) = \mathcal{E}$ is a torsion line bundle. Now, by the main theorem in \cite{K} below, the line bundle $\mathcal{O}_\mathcal{C} (D+D')$ must be non-special.\\

\begin{enumerate}
\item[$(\ddagger)$] {\bf Theorem (G. R. Kempf, \cite{K}).} Assume that $\mathcal{C}$ is a smooth complete algebraic
curve with general moduli in characteristic zero. Let $D_0$ and $D_{\infty}$ be
distinct effective divisors of $\mathcal{C}$ such that
$\mathcal{O}_{\mathcal{C}} (D_0 - D_{\infty})$ is a torsion line
bundle. Then the cohomology $H^1 (C,\mathcal{O}_{\mathcal{C}} (D_0 +
D_{\infty}))$ must be zero.\\
\end{enumerate}

\noindent Thus we have
\begin{equation*}
h^0 (\mathcal{C},\mathcal{O}_\mathcal{C} (D+D')) =  6+1-g \geq 1
\end{equation*}
and hence $g \leq 6$, which contradicts to our assumption that $g
\geq 7$. Therefore it is shown that either $I(\mathcal{C}_1 )$ or else $I(\mathcal{C}_2 )$ must be generated by quadrics. Since
\begin{equation*}
\mathcal{L}   =  \mathcal{M}^2 = (\mathcal{M} \otimes \mathcal{E} )^2 ,
\end{equation*}
it follows by Remark \ref{rmk:curve important}.(2) that $(\mathcal{C},\mathcal{L})$ satisfies property $\QR (3)$.   \qed \\

\end{document}